\documentclass[11pt]{article}
\usepackage{amsmath}
\usepackage{amssymb}
\usepackage{latexsym}
\usepackage{amsthm}
\usepackage{fullpage}
\usepackage[colorinlistoftodos]{todonotes}
\usepackage{marginnote}
\usepackage{enumitem}
\usepackage{mathtools}
\usepackage{dsfont}
\usepackage{lipsum}

\newcommand\blfootnote[1]{%
  \begingroup
  \renewcommand\thefootnote{}\footnote{#1}%
  \addtocounter{footnote}{-1}%
  \endgroup
  }

 \numberwithin{equation}{section}

\newtheorem{example}{Example}[section]
\newtheorem{theorem}{Theorem}[section]
\newtheorem{lemma}{Lemma}[section]

\newtheorem{remark}[example]{Remark}
\newtheorem{definition}[example]{Definition}

\definecolor{vine}{rgb}{0.7,0.1,0.1}

\newcommand{\R}{{\mathbb R}}
\newcommand{\N}{{\mathbb N}}
\newcommand{\Z}{{\mathbb Z}}

\newcommand{\C}{{\mathbb C}}
\newcommand{\eps}{\epsilon}
\renewcommand{\d}{{ \rm d}}

\newcommand{\ra}{\right\rangle}
\newcommand{\la}{\left\langle}

\newcommand{\no}[2]{ \left\| #1 \right\|_{#2} }

\renewcommand{\div}{{\operatorname{div}}}

\newcommand{\To}{{\mathbb{T}^2}}
\newcommand{\weak}{\rightharpoonup}
\newcommand{\weaks}{\rightharpoonup^*}
\newcommand{\Sp}{{\mathbb{S}}}
\newcommand{\Ld}{{\dot{L}}}
\newcommand{\Wd}{{\dot{W}}}

\newcommand{\loc}{{\operatorname{loc}}}

\author{Joshua Kortum\footnote{Institute of Mathematics, University of W\"urzburg, Emil-Fischer-Str.\ 40, 97074 W\"urzburg, Germany (joshua.kortum@mathematik.uni-wuerzburg.de)}}
\title{Concentration-cancellation in the Ericksen-Leslie model}

\begin{document}
\maketitle

\begin{abstract}
We establish the subconvergence of weak solutions to the Ginzburg-Landau approximation to global-in-time weak 
solutions of the Ericksen-Leslie model for nematic liquid crystals on the torus $\To$. The key argument is a variation  of concentration-cancellation methods originally introduced by DiPerna and Majda  to investigate the weak stability of solutions to the (steady-state) Euler equations. 
\end{abstract}


\blfootnote{
 AMS-Classification:  Primary 35Q35, 35K55; secondary 76A15, 58E20}

\section{Introduction}

The Ericksen-Leslie model describes the motion of  nematic liquid crystal flows \cite{ericksen,leslie}. The nematic phase of liquid crystals can be thought of as an intermediate state of isotropic flow and a solid crystalline phase where the rod-like molecules do not act freely  but tend to align in a certain direction. In order to depict this behaviour, two quantities are used to model the liquid crystal,  the velocity $v$ and the unitary director field $d$ which represents  the orientation of the molecules in space. 

In this article, we study the variant of the Ericksen-Leslie model proposed  in \cite{lin1989}, which  reads
\begin{align} \label{ericksen-leslie}
\begin{cases} 
	 \partial_t v + (v \cdot \nabla ) v + \nabla p - \Delta v = - \div ( \nabla d \odot \nabla d), \\
	 \div \, v = 0, \\
	 \partial_t d + (v \cdot \nabla ) d =  \Delta d + |\nabla d|^2 d , \qquad |d|\equiv 1,
	 \end{cases}
\end{align}
(all  physical constants set  to one)                                                                                                                                                                                                                                                                                                                                                                                                                                                                                                                                                                                                                                                                                                                                                                                                                                                                                                                                                                                                                                                                                                                                                                                         on the space-time domain $ \To \times [0,T]$ for any given $T>0$. Here,  $p: \To \times [0,T] \to \R$ denotes the 
underlying pressure and serves as a Lagrange multiplier subject to the incompressibility  condition $\div \, v =0$. 
The system is supplemented by initial data $(v_0,d_0): \To \to \R^2 \times \Sp^2$,  $\div \, v_0 =0$. We refer to \cite{hieber} 
for modelling issues of the general Ericksen-Leslie equations as well as analytical aspects. The energetic variational approach is executed  in e.g. \cite{deanna,lin1995}.

From the mathematical point of view, \eqref{ericksen-leslie} preserves the major mathematical challenges of the full dynamic Ericksen-Leslie model. Indeed, even on two-dimensional domains, the  system \eqref{ericksen-leslie} might  form singularities in finite-time as shown by Huang et al. in \cite{lin2016}. Establishing existence or uniqueness of solutions is therefore a non-trivial problem mainly due to the harmonic map heat flow-like equation for the director field $d$. In order to construct solutions to \eqref{ericksen-leslie}, one would like to use an approximation scheme and pass to the limit with the help of  a-priori estimates or a more refined analysis. However, the associated energy law, see \eqref{energyLaw} below, does not provide strong enough bounds,  and the main problem consists of the limit passage on 
the right-hand side in the momentum equation. Lin, Lin and Wang \cite{lin} and Hong \cite{hong} first proved the existence of weak solutions to \eqref{ericksen-leslie} on a two-dimensional bounded domain or $\R^2$. Both relied on Struwe's  \cite{struwe} construction of partially regular solutions of the harmonic map heat flow in two dimensions.  To this end, Hong constructed local-in-time smooth solutions to \eqref{ericksen-leslie} via the Ginzburg-Landau approximation 
\begin{align} \label{ginzburg-landau}\displaystyle
\begin{cases} 
	 \partial_t v_\eps + (v_\eps \cdot \nabla ) v_\eps + \nabla p_\eps - \Delta v_\eps = - \div ( \nabla d_\eps \odot \nabla d_\eps), \\
	 \div \, v_\eps = 0, \\
	 \partial_t d_\eps + (v_\eps \cdot \nabla ) d_\eps =  \Delta d_\eps + \frac{1}{\eps^2}(1-|d_\eps|^2)d_\eps 
	 \end{cases}
\end{align}
with $\eps \to 0^+$. System \eqref{ginzburg-landau} depicts a well established approximation to \eqref{ericksen-leslie} (cf. \cite{hieber}) in order to circumvent the above mentioned difficulties.  Lin and Liu  \cite{lin1995} first showed  existence of weak solutions  as well as strong solutions to \eqref{ginzburg-landau}
on a bounded domain either in two dimensions or under a smallness condition on the initial data in dimension three for fixed $\eps>0$. The energy related to \eqref{ginzburg-landau} reads
$$E_\eps (d_\eps) = \frac12 \int_\To |v_\eps|^2 + |\nabla d_\eps|^2 + \frac{1}{2\eps^2}(1-|d_\eps|^2)^2,$$
where the second term  penalizes variations from the constraint $|d_\eps|\equiv 1$.
 As 
$\eps$ tends to zero, the director field is forced to attain values in the sphere, i.e.\ $|d_\eps| \to 1$. Thus one expects
convergence of $d_\eps$  to solutions of \eqref{ericksen-leslie}. Indeed, this fact is proven for strong solutions 
locally in time in \cite{feng,hong}. However, an extension of this strong convergence result to larger times is not possible due to blow-up of some solutions to \eqref{ericksen-leslie}.

In this work, we actually prove the subconvergence of weak solutions to \eqref{ginzburg-landau} to weak solutions
of \eqref{ericksen-leslie} globally in time. In \cite[p.\ 1108]{hieber} and \cite[p.\ 290]{ball2017}  this issue was highlighted as an open problem.
This limit passage is also of interest for numerical approximations \cite{walkington}, in the stochastic Ericksen-Leslie system \cite{debouard} or the flow of magnetoviscoelastic materials (see \cite{schloemerkemper2018}). The singular limit problem $\eps \to 0^+$ for the harmonic map heat flow into spheres and more general manifolds was first studied by Chen and Struwe in \cite{chen,chen2} (see also \cite{alouges} for the related Landau-Lifshitz equation). In the Ericksen-Leslie model, the difficulty is to pass to the limit in the stress tensor $-\div (\nabla d \odot \nabla d)$ as long as one is restricted to the energy estimate \eqref{energyLaw}. In general, $\nabla d_\eps \odot \nabla d_\eps \weaks \nabla d \odot \nabla d + \eta$ for a possibly non-vanishing matrix-valued measure $\eta$. We will not show $\eta =0$ but  use the idea of concentration-cancellation for Euler equations introduced by DiPerna and Majda in \cite{diperna} (see also \cite{majda2002}) to verify that the weak limit  $(v,d)$ fulfills \eqref{ericksen-leslie}. This procedure becomes possible since $\div (\nabla d \odot \nabla d)$ enjoys
the same structure as the convective term in the Euler equations
\begin{align*}
	\begin{cases} \partial_t v + \div(v \otimes v) + \nabla p =0,\\
	\div \, v =0. \end{cases}
\end{align*}

While this technique is successful here, we remark that it cannot be used to prove existence of weak solutions to the time-dependent Euler equations
in general. One main problem depicts the low regularity of $\partial_t v$. However, there exist cases where concentration-cancellation occurs,
see e.g.\ the result of Delort \cite{delort} for non-negative vorticities bounded in the space of measures (see also \cite{evans}), or if certain  assumptions on the time-derivative of $v$ \cite{schochet} or the size of the defect measure are made \cite{diperna,schochet}. In contrast to the delicate situation for the Euler equations, enough regularity of $\partial_t v$ in \eqref{ginzburg-landau} is available, see \eqref{v_tReg}, such that we can perfom the limit passage without further assumptions. Indeed, it turns out to be sufficient to stick to the initial idea of \cite{diperna}.

A crucial idea of the proof is to fix a time  $t \in [0,T]$ and then to carry out  a concentration-cancellation argument in the limit passage. In particular, we show that  $(\nabla d_\eps(t))_\eps$ may concentrate only in a finite number of points. This result is in correspondence with well-known results for approximated harmonic maps, cf.~\cite{lin1999,qing,wang}. From the smallness of the concentration set we conclude that the limit of \eqref{ginzburg-landau} satisfies \eqref{ericksen-leslie} in the weak sense. 
The method of fixing a time step  is inspired by \cite{lin2}, where Lin and Wang consider the three-dimensional liquid crystal flow in the special case of  solutions with values in the upper half-sphere $d(x,t) \in \Sp^2_+$. We carry out this
program on the space domain $\To$, which allows to use the Fourier expansion.

The structure of the paper is as follows:
In the second section, basic notation and the main results, Theorems 2.1 and 2.2, are stated. The third section deals with the proofs of these theorems. The proofs are divided into several steps: Section 3.1  is devoted to a-priori estimates, Sections 3.2 and 3.3 provide an $\varepsilon_0$-regularity statetment and an estimate on the concentration set in space and Section 3.4 concludes the proofs  with the limit passage explained above.

\section{Setting and results}

Defining $A \odot B:= A^\top B $, we investigate the initial  value problem
\begin{align} 
	& \partial_t v + (v \cdot \nabla ) v + \nabla p - \Delta v = - \div ( \nabla d \odot \nabla d), \label{equationNS}\\
	& \div \, v = 0, \\
	& \partial_t d + (v \cdot \nabla ) d =  \Delta d + |\nabla d|^2 d, \qquad |d|\equiv 1 \label{equationLLG}
\end{align}
on $\To \times [0,T]$ with $\To=(\R / 2\pi\Z)^2$ and $T>0$ given.  The prescribed initial data consist of 
\begin{alignat}{3}
	& v(x,0) 	 =  v_0(x), \, \div \,  v_0 = 0 	& \text { on } \To \times \{0\} \label{ini1} \\
	& d(x,0)	 = d_0(x),	\, |d_0|\equiv 1 		& \text { on } \To \times \{0\}. \label{ini4}
\end{alignat}
On $\To$, we may write  $ f \in L^2(\To,\R^2)$ as  Fourier expansion $f = \sum_{ k \in \Z^2} \hat{f_k} e^{ik\cdot (\cdot)}$. The homogeneous space of square-integrable functions is denoted by $\dot{L}^2(\To, \R^2)$ 
as well as $\dot{W}^{1,2}(\To,\R^2)$  for the homogeneous Sobolev space.  We use  $X_\div$ for (weakly) solenoidal functions in the function space $X$ (e.g.\ for $X= L^2, W^{1,p}, C^\infty...$). \\
Note that it makes sense to consider $v \in \dot{L}^2_\div$ as a solution to \eqref{equationNS} whereas $d$ is rather considered to be an element of  the nonhomogeneous space $W^{1,2}(\To,\Sp^2)$ due to the constraint $|d|\equiv 1$. It is  useful to represent $ f \in \dot{L}_\div^2(\To, \R^2)$   as
$$ f= \nabla^\perp g,$$
where $ \nabla^\perp = (- \partial_2, \partial_1)^\top $  and $g \in W^{1,2}(\To)$. That this is possible is easily seen by Fourier expansion ($f \in \dot{L}_\div^2$ implies $k \cdot \hat{f_k} = 0$ for all $k\in \dot{\Z}^2$, which in turn implies  $\hat{f_k}= (-k_2,k_1)^\top \lambda_k$ for some $\lambda_k \in \C$ and all $k \in \dot{\Z}^2=\Z^2\backslash \{ (0,0)^\top\}$).

For some Banach space $X$, the time-dependent Bochner spaces are denoted by $L^p(0,T; X)$ or $W^{1,p}(0,T; X)$ 
respectively and we use  $\no{\cdot}{L^p(0,T; X)}  = \no{\cdot}{L^p_t X_x}$ short hand for the norm.

Now we are in the position to define  a weak solution to \eqref{equationNS}--\eqref{ini4}:
\begin{definition}	\label{defWeak}
Let $T>0$. A pair 
\begin{align*}
	v \in &L^\infty(0,T; \dot{L}^2_\div(\To,\R^2)) ~  \cap ~ L^2(0,T; \dot{W}_\div^{1,2}(\To,\R^2)),  \\
  	d \in  &L^\infty (0,T; W^{1,2}(\To,\Sp^2)) ~ \cap  ~W^{1,2}(0,T;L^{4/3}(\To,\R^3))
\end{align*}
is called a weak solution
to the initial value problem \eqref{equationNS}--\eqref{equationLLG} subject to the initial conditions \eqref{ini1}--\eqref{ini4} if 
\begin{align*}
& \int_0^T \int_\To - v  \cdot\partial_t \phi - v \otimes v : \nabla \phi + \nabla v : \nabla \phi - \nabla d \odot \nabla d : \nabla \phi \, \d x \,\d t  = \int_\To v_0 \cdot  \phi \, \d x , \\
& \int_0^T \int_\To \partial_t d \cdot \xi + (v \cdot \nabla ) d  \cdot \xi + \nabla d : \nabla
\xi - |\nabla d|^2 d \cdot \xi \, \d x \,  \d t  = 0
\end{align*}
holds true for all $\phi \in C_{\div}^\infty( \To \times [0,T], \R^2)$ and $ \xi \in C^\infty( \To \times [0,T], \R^3)$ with $\phi(T), \xi(T) = 0$. Additionally, $(v,d)$  attend the initial data $(v_0,d_0) \in L^2_{\div}(\To,\R^2) \times W^{1,2}(\To, \Sp^2)$  in the weak sense, i.e.
\begin{align*}
	\int_\To v(t) \cdot \psi \,  \d x \to \int_\To v_0\cdot  \psi \, \d x, \qquad 
	\int_\To \nabla d(t) : \zeta \,  \d x \to \int_\To \nabla d_0 :\zeta \, \d x
\end{align*}
for all $\psi \in C_{\div}^\infty( \To, \R^2)$ and  $ \zeta \in C^\infty( \To, \R^{3\times 2})$ as $t \to 0^+$.
\end{definition}
%
Our notion of weak solutions resembles the usual definition of weak harmonic map heat flows (see \cite{chen2}) and weak Navier-Stokes flows (see \cite{robinson2016}).  
In contrast to previous works, we  construct these weak solutions out of weak solutions to the Ginzburg-Landau
approximation. The latter reads 
\begin{align} \label{approx1}
	& \partial_t v_\eps + (v_\eps  \cdot \nabla ) v_\eps +\nabla p_\eps -  \Delta v_\eps = -  \div(\nabla d_\eps \odot \nabla d_\eps), 
\qquad 	 \div \, v_\eps = 0, \\
	& \label{approx2} \partial_t d_\eps +(v_\eps \cdot \nabla ) d_\eps =  \Delta d_\eps + \dfrac{1}{\eps^2} (1-|d_\eps|^2) d_\eps, \\
	& v_\eps(\cdot, 0) =v_0, \quad d_\eps(\cdot, 0 ) = d_0 	\label{iniApprox}
\end{align}
for $0<\eps \leq 1$. Weak solutions $(v_\eps,d_\eps)$  to \eqref{approx1}--\eqref{iniApprox}   are known to exist globally in time. Here, we eludicate the limiting behavior as $\eps \to 0^+$. Formally, $(v_\eps, d_\eps)$  converge to solutions of \eqref{equationNS}--\eqref{ini4}. We give a precise meaning to this idea with 
the following result: 

\begin{theorem} 	\label{main-thm:one}
Let $(v_\eps,d_\eps)_{0<\eps\leq1}$ be the family of unique weak solutions to \eqref{approx1}--\eqref{iniApprox} for $(v_0,d_0) \in \dot{L}^2_\div (\To,\R^2) \times  W^{1,2}(\To, \Sp^2)$.
 Then there exists a subsequence $(\eps_j)_j$ with $\lim_{j \to \infty} \eps_j = 0^+$ such that 
\begin{align*}
(v_{\eps_j}, d_{\eps_j}) \weaks (v,d) 
\end{align*}
in $L^\infty(0,T;L^2_\div(\To,\R^2)) \times L^\infty(0,T; W^{1,2}(\To,\R^3))$  as well as pointwise a.e.~on $\To \times [0,T]$ with $(v,d)$ being a weak solution to 
 \eqref{equationNS}--\eqref{ini4} in the sense of Definition 2.1.
\end{theorem}

Theorem 2.1 provides a new argument to establish existence of weak solutions to the Ericksen-Leslie model which 
satisfy the physically reasonable energy inequality:

\begin{theorem} \label{main-thm:two}
Suppose $v_0 \in \dot{L}^2_\div (\To,\R^2)$ and $d_0 \in W^{1,2}(\To, \Sp^2)$. Then there exists a weak solution $(v,d)$ in the sense of Definition \ref{defWeak} to the system \eqref{equationNS}--\eqref{ini4} that satisfies the energy inequality, i.e.
\begin{align*}
\int_\To |v|^2(t) + |\nabla d|^2(t)\, \d x + 2\int_0^t \int_\To |\nabla v|^2 + \left| \Delta d + |\nabla d|^2 d \right|^2 \, \d x \, \d t
\leq \int_\To |v_0|^2 + |\nabla d_0|^2 \d x
\end{align*}
is valid for almost all $t \in[0,T]$.
\end{theorem}

\begin{remark}[Uniqueness] 
As for the harmonic map heat flow, we do not know whether the solution in Theorem~2.2 is a solution in the sense of Struwe in \cite{struwe}. In particular, the energy is not known to be nonincreasing. For the harmonic map heat flow,
Bertsch et al.\  \cite{bertsch}  and Topping \cite{topping}  proved   the existence of infinitely many weak solutions with conserved but increasing
energy at certain time steps. The same behaviour may be possible in our case. 
\end{remark}

\begin{remark}[Stability with respect to initial data]
It can easily be checked that Theorem \ref{main-thm:one} remains true for a sequence of initial data
$$ (v_0^\eps, d_0^\eps)_\eps \to (v_0, d_0) $$
strongly in $\dot{L}^2_\div(\To,\R^2) \times W^{1,2}(\To,\Sp^2)$. However, if only weak convergence is given, a result like Theorem \ref{main-thm:one} may not be available in general since various oscillation and concentration effects occur. 
\end{remark}

\section{Proofs of Theorem 2.1 and 2.2}

\subsection{A-priori estimates}

This section is devoted to establishing the Ginzburg-Landau approximation and the collection of (mostly standard) a-priori estimates. As for most systems arising from physics, this is done by employing the energy law associated to the system and secondary estimates on the time derivatives via duality.

In \cite{lin1995} (see also \cite{schloemerkemper2018}), Lin and Liu showed  that global-in-time weak solutions to the Ginzburg-Landau approximation exist for initial data\footnote{In \cite{lin1995}, $d_0 \in W^{3/2,2}(\partial \Omega)$ is required for a smooth domain. This is omitted since $\partial \To = \emptyset$.} 
$(v_0,d_0) \in L^2_\div( \To, \R^2) \times W^{1,2}(\To, \Sp)$ on a smooth bounded domain $\Omega\subset \R^2$. The result was proven 
by a Galerkin approximation scheme and  carries over to the case $\Omega = \To$. More precisely, there is a unique pair
\begin{align*}
	v_\eps \in &L^\infty(0,T; {L}^2_\div(\To,\R^2)) ~  \cap ~ L^2(0,T; {W}_\div^{1,2}(\To,\R^2)),  \\
  d_\eps \in  &L^\infty (0,T; W^{1,2}(\To,\Sp^2)) ~  \cap ~  L^2(0,T;W^{2,2}(\To,\R^3))
\end{align*}
solving \eqref{approx1}--\eqref{iniApprox} in the weak sense. This regularity suffices  
to perfom  typical calculations yielding energy estimates and a-priori bounds. Multiplying  \eqref{approx1} by $v$, \eqref{approx2} by $-\Delta d - \frac{1}{\eps^2}(1-|d|^2)d$ and integrating over $\To \times [0,t]$, we obtain 
\begin{equation} \label{energyLaw}
\begin{split}
\int_\To |v_\eps(t)|^2 + & |\nabla d_\eps(t)|^2 + \frac{1}{2\eps^2} (1-|d_\eps(t)|^2)^2   + 2 \int_0^t \int_\To |\nabla v_\eps|^2 + \left|\Delta d_\eps  + \frac{1}{\eps^2} (1-|d_\eps|^2)d_\eps\right|^2 \\
&=
 \int_\To |v_0|^2 + |\nabla d_0|^2 
=:2 E_0.
\end{split}
\end{equation}
Here we benifited from the fact that $|d_0|\equiv 1$ almost everywhere. Further,  $ d_\eps$  enjoys a maximum principle (see \cite{alouges,lin2}):

\begin{lemma} \label{maximumPrinciple}
Suppose $(v_\eps, d_\eps)$ is a solution to \eqref{approx1}--\eqref{iniApprox}. Then $d_\eps$ satisfies
$$ | d_\eps(x,t)|\leq 1$$
for almost every $(x,t) \in \To \times [0,T]$.
\end{lemma}
\begin{proof}
For $k \in \N$ we define the auxiliary function $h_\eps^k :\To \times [0,T]\to \R$ by
\begin{align*}
	h_\eps^k(x,t) = \begin{cases}
	k^2-1 & \text{for } k < |d_\eps(x,t)|, \\
	|d_\eps(x,t) |^2 -1  & \text{for } 1 < |d_\eps(x,t)| \leq k, \\
	0 & \text{for } |d_\eps(x,t)| \leq 1.
	\end{cases}
\end{align*}
By \eqref{approx2}, we have 
\begin{align*}
\partial_t h_\eps^k + v_\eps \cdot \nabla h_\eps^k & =  \Delta h_\eps^k -2 \chi_{\{1 < |d_\eps| \leq k\} } \left( |\nabla d_\eps|^2 + \frac1{\eps^2} (|d_\eps|^2 -1) |d_\eps|^2) \right) \\
& \leq \Delta h_\eps^k
\end{align*}
in the weak sense. Next we multiply the differential inequality by $h_\eps^k$,  and an integration by parts yields (due to the periodicity of $\To$ and $|d_0| \equiv 1$)
\begin{align*}
\frac12 \int_\To |h_\eps^k(t)|^2 + \int_0^t \int_\To |\nabla h_\eps^k|^2 \leq - \int_0^t \underbrace{\int_\To v_\eps \cdot \frac{\nabla}{2}|h_\eps^k|^2}_{=0} = 0.
\end{align*}
This can only be true if $h_\eps^k = 0$ a.e. on $\To \times [0,T)$ and therefore the assertion follows.
\end{proof}

The energy law \eqref{energyLaw} and Lemma \ref{maximumPrinciple} yield  a-priori bounds
\begin{align*}
	 \no{v_\eps}{L^\infty_t \Ld^2_x} 	& \leq C, \\
	\no{v_\eps}{L^2_t \Wd^{1,2}_x} 	& \leq C, \\
	\no{1-|d_\eps|^2}{L^\infty_t L^2_x} 	& \leq C\eps, \\
	\no{d_\eps}{L^\infty_{t,x}} 	& \leq 1, \\
	\no{\nabla d_\eps}{L^\infty_t L^2_x} 	& \leq C, \\
	 \no{\Delta d_\eps +\frac{1}{\eps^2}(1-|d_\eps|^2)d_\eps }{L^2_t L^2_x } 	& \leq C,
\end{align*}
 uniformly in $\eps >0$. Ladyzhenskaya's inequality also implies 
$$ \no{v_\eps}{L^4_tL^4_x} 	 \leq C.$$
In order to achieve strong convergence we make use of the generalized Aubin-Lions lemma \cite[Lemma 7.7]{roubicek} for which some bounds on the time derivatives of $(v_\eps, d_\eps)$ are needed. The estimates above and \eqref{approx2} allow us to deduce 
$$ \no{\partial_t d_\eps}{L^2_t L^\frac43_x} \leq C.$$
Considering $\partial_t v_\eps$, we first note that for $\phi \in C_{\div}^\infty(\To,\R^2)$ one has
$$\frac{1}{\eps^2} \int_\To (\nabla d_\eps)^\top  (1-|d_\eps|^2)d_\eps  \cdot \phi = - \frac{1}{4\eps^2} \int_\To \nabla (1-|d_\eps|^2)^2  \cdot \phi =0. 				 $$
Secondly, we employ the identity $\div (\nabla d_\eps \odot \nabla d_\eps) = \nabla \frac{|\nabla d_\eps|^2}{2} + (\nabla d_\eps)^\top \Delta d_\eps$.
Now testing   \eqref{approx1} by  $\phi \in C_{ \div}^\infty(\To\times [0,T], \R^2)$ gives rise to 
\begin{align*}
&\left|\int_{\To\times[0,T]}  \right.  \partial_t v_\eps \cdot   \phi \bigg| \\
				&  \leq \left| \int_{\To\times[0,T]}  v_\eps \otimes v_\eps :\nabla \phi \right|  + \left| \int_{\To\times[0,T]}  \nabla v_\eps : \nabla \phi \right| 
				 + \left| \int_{\To\times[0,T]}  (\nabla d_\eps)^\top \left( \Delta d_\eps + \frac{1}{\eps^2} ( 1- |d_\eps|^2)d_\eps \right) \cdot \phi \right| \\
				  &\leq \left( \no{v_\eps}{L^4_{t} L^4_x}^2+ \no{v_\eps}{L^2_t \Wd^{1,2}_x} 
				+ \no{\nabla d_\eps}{L^\infty_t L^2_x} \cdot  \no{\Delta d_\eps + \frac{1}{\eps^2}(1-|d_\eps|^2)d_\eps}{L^2_{t} L^2_x}\right) \times \\
			& \,\,\,\,\,\,\,\,  \times  \left( \no{\phi}{L^2_t W^{1,2}_\div} +	\no{\phi}{L^2_t C_\div}\right).
\end{align*}
Since $X_s :=W^{1,s}_{\div}(\To,\R^2) \subset W^{1,2}_\div (\To ,\R^2) \cap C_{\div}(\To ,\R^2) $ is true for any $s>2$, the estimate
\begin{equation} \label{v_tReg}
\no{\partial_t v_\eps }{L^2_t X_s^{*}} \leq C
\end{equation} 
is valid independently of $\eps >0$. We point out that we benefited from the interplay of solenoidal test functions and the  gradient flow structure of the system to improve the control  in space of $\partial_t v_\eps$ compared to the standard  estimate $\partial_t v_\eps \in L^2 (0,T; (W_\div^{1,\infty}(\To,\R^2))^\ast )$. This is in sharp contrast to the framework 
of the Euler equations and allows us to use the concentration-cancellation techniques from \cite{diperna} later on. \\
As a consequence of the above estimates, we can choose a subsequence $(\eps_i)_{i \in \N} \subset (0,1]$ with $\lim_{i \to \infty} \eps_i = 0^+$ such that 
\begin{align}
\label{conv1}v_{\eps_i} \to v 	\qquad 											&	 \text{in } L^2(\To \times [0,T], \R^2) \text{ and a.e.}, \\
						\nabla v_{\eps_i} \weak \nabla v 	\qquad 				& \text{in } L^2(\To \times [0,T], \R^{2 \times 2}), \\
\label{conv3}						\partial_t v_{\eps_i} \weak \partial_t v	\qquad 	& \text{in } L^2(0,T;X_s^*) \text{ for } s>2, \\
\label{conv4}d_{\eps_i} \to d \qquad 													& \text{in } L^p( \To \times [0,T],\R^3) \text { for any } p\in (1, \infty)
																																\text{ and a.e.}, \\
\label{conv5}|d_{\eps_i}|^2 \to 1 \qquad 											& \text{in } L^\infty(0,T; L^1(\To)) 
																																\text{ and a.e.}, \\
\label{conv6}\nabla d_{\eps_i} \weaks  \nabla d \qquad				& \text{in } L^\infty(0,T; L^2(\To,\R^{3 \times 2})), \\
\label{conv7}	\partial_t d_{\eps_i} + (v_{\eps_i} \cdot \nabla ) d_{\eps_i}  \weak \partial_t d+ (v \cdot \nabla) d
						\qquad 																					& \text{in } L^2( \To \times [0,T], \R^3).
\end{align}
Additionally, we  can choose the subsequence  such that
\begin{align}
\label{conv8}		v_{\eps_i}(\cdot,t)  \to v(\cdot,t) 	\qquad 				&	 \text{in } L^2(\To,\R^2) \text{ for a.a. } t \in [0,T], \\
\label{conv9}		d_{\eps_i}(\cdot,t)  \to d(\cdot,t) 	\qquad 				&	 \text{in } L^2(\To, \R^3) \text{ for a.a. } t \in [0,T]. 
\end{align}

\subsection{$\varepsilon_0$-regularity}

According to the idea of fixing  certain time steps in $[0,T]$ and passing to the limit, we
consider the equation
\begin{align} \label{appr harm}
\Delta  u_\eps + \frac1{\eps^2}(1-|u_\eps|^2)u_\eps = \tau_\eps
\end{align}
on $\To$ for some $\tau_\eps \weak \tau$ in $L^2(\To,\R^3)$ and $u_\eps \in W^{2,2}(\To,\R^3)$ being a strong solution for $\eps >0$. This is motivated by equation \eqref{approx2} where taking $u_\eps= d_\eps(t)$ (at first formally) for  fixed $t \in [0,T]$ leads to this situation. As $\eps \to 0^+$, we have a singular limit problem and we expect that $u_\eps$ converges to an approximated harmonic map $u : \To \to \Sp^2$, i.e. 
$$ \Delta u + |\nabla u|^2 u = \tau - (\tau \cdot u ) u $$
in some sense. Strong convergence of  $(u_\eps)_\eps$ in $W^{1,2}$ cannot be expected even in two dimensions (see e.g.\ \cite{bethuel1994,lin1999}). 
Using the general idea of partial regularity for elliptic equations, we obtain strong convergence of $u_\eps$ in $W^{1,2}$ except of a finite set  in $\To$. 

Inspired by \cite{lin1999,lin2}, we prove an  $\eps_0$-regularity theorem which leads to locally strong convergence of $(u_\eps)_\eps$ in $W^{1,2}$. Because of the derived energy estimate \eqref{energyLaw}, we assume 
\begin{equation} \label{energyBound}
\sup_{0< \eps \leq 1} \int_\To \frac12 |\nabla u_\eps|^2 + \frac{1}{4\eps^2}(1-|u_\eps|^2)^2 \leq E_0.
\end{equation}

\begin{theorem}[$\eps_0$-regularity] \label{epsRegularity}
Suppose that $(u_\eps)_\eps$ is a sequence of strong solutions to (\ref{appr harm}) with $ 0 < \eps \leq1$ satisfying \eqref{energyBound}. Further, let $\tau_\eps \weak \tau$ in $L^2(\To, \R^3)$ for $\eps \to 0^+$  and $|u_\eps|\leq 1$ for $\eps >0$.
Then there exists an $\varepsilon_0 >0$  such that if for $x_0 \in \To$ 
\begin{align*}
\sup_{0< \eps \leq 1} \int_{B_{r_1}(x_0)} \frac12 |\nabla u_\eps|^2 + \frac1{4\eps^2}(1-|u_\eps|^2)^2 \leq \varepsilon_0^2
\end{align*}
holds true for some $r_1>0$, there exists a subsequence\footnote{not relabeled} with $u_\eps \to u$ strongly in $W^{1,2}(B_{r_1/4}(x_0), \R^3)$ for $\eps \to 0^+$.
\end{theorem}

\begin{proof} Let $x_0 =0$ without loss of generality. The proof is divided into four steps.\\
\textit{Step 1:} We show that
 $$|u_\eps(x)-u_\eps(y)|\leq C \left( \frac{|x-y|}{\eps}						\right)^{1/2}
 \quad \text{on } B_{r_1/2}(x_1)$$   for any $x_1 \in B_{r_1/2}$.  Indeed, introducing the  scaled solution  $\hat{u_\eps}(x) = u_\eps (x_1 + \eps x)$ yields 
 $$ \Delta \hat{u_\eps} = - (1 - |\hat{u_\eps}|^2) \hat{u_\eps} + \hat{\tau_\eps}$$
on $B_{r_1/(2\eps)}$ for $\hat{\tau_\eps}(x) = \eps^2 \tau_\eps (x_1 + \eps x)$. Using elliptic theory \cite[Theorem 9.9]{gilbarg}, we obtain the estimate
  $$ \no{\hat{u_\eps}}{W^{2,2}(B_{r_1/(2\eps)})} \lesssim1 +  \no{\hat{\tau_\eps}}{L^2(B_{r_1/\eps})} \leq C$$
 for every $\eps \in (0,1]$. Now the Morrey embedding $W^{2,2}\hookrightarrow 			C^{1/2}$ and  rescaling back imply the assertion.\\
\textit{Step 2:} We use the H\"older continuity to show that $|u_\eps(x)| \geq \frac12$ on  $B_{r_1/2}$.   On the contrary, assume there existed some $x_1\in B_{r_1/2}$ with $|u_\eps(x_1)|< 1/2$. Because of the H\"older estimate above we have, for $ x \in B_{ \eps \theta_0}(x_1)$, that
 $$ |u_\eps (x)|\leq \frac34 $$
 provided $0<\theta_0 < \frac1{16C^2}$. Therefore it follows
 \begin{align*}
 	\int_{B_{\theta_0\eps}(x_1)} \frac{(1-|u_\eps|^2)^2}{4 \eps^2} \geq
 	\left(\frac{7}{16}\right)^2 \frac{\theta_0^2 \eps^2 \pi}{4 \eps^2}= \left(\frac{7}{32}\right)^2 \theta_0^2 \pi
 \end{align*}
 which contradicts the assumption that
 \begin{align*}
 	\int_{B_{\theta_0\eps}(x_1)} \frac{(1-|u_\eps|^2)^2}{4 \eps^2} \leq 
 	\int_{B_{r_1}(0)} \frac{1}{2} |\nabla u_\eps |^2 + \frac{1}{4 \eps^2} (1-|u_\eps|^2)^2 \leq \varepsilon_0^2
 \end{align*}
 for a chosen sufficiently  small $\varepsilon_0 >0$.\\
\textit{Step 3:} We use $|u_\eps| \geq \frac12$ on $B_{r_1/2}$ to engage  the polar decomposition
 $$ u_\eps = |u_\eps| \frac{u_\eps}{|u_\eps|} =: \rho_\eps \psi_\eps.$$ 
Notice that $|\psi_\eps| \equiv 1$ as well as
 $$|\nabla \psi_\eps | + |\nabla \rho_\eps| \lesssim |\nabla u_\eps | \lesssim 
 |\nabla \psi_\eps | + |\nabla \rho_\eps|.$$
Multiplying \eqref{appr harm} by $\psi_\eps$  and applying the multiplication operator $\frac{1}{\rho_\eps} ((\cdot)  - ( \psi_\eps \cdot (\cdot) ) \psi_\eps )$
to \eqref{appr harm}, we obtain the system of equations
\begin{align}
 	& \Delta \rho_\eps + \frac1{\eps^2} \rho_\eps (1-\rho_\eps^2) - \rho_\eps 
 	|\nabla \psi_\eps |^2 =\tau_\eps \psi_\eps =:g_\eps \label{absoluteValueEquation} \\
 	& \Delta \psi_\eps = -  |\nabla \psi_\eps|^2 \psi_\eps - \frac{2}{\rho_\eps}\nabla \psi_\eps  \nabla \rho_\eps  
 	+ \frac{1}{\rho_\eps} (\tau_\eps - (\tau_\eps \psi_\eps) \psi_\eps)=:f_\eps
 \end{align}
on $B_{r_1/2}$ respectively. Considering the second equation, we again  employ estimates from the theory of elliptic equations 
\cite[Theorem 9.9]{gilbarg} by
$$ \no{\nabla^2 \psi_\eps}{L^\frac43} \lesssim \no{f_\eps}{L^\frac43} 
	\lesssim (\no{\nabla \psi_\eps}{L^2} + \no{\nabla \rho_\eps}{L^2}) 
	 \no{\nabla \psi_\eps}{L^4} + \no{\tau_\eps}{L^2} .
$$
Secondly the Sobolev imbedding gives $\no{\nabla \psi_\eps}{L^4} \lesssim \no{\nabla^2 \psi_\eps }{L^\frac43} +1$. 
We use the assumption $\no{\nabla u_\eps}{L^2}\leq \sqrt{2} \eps_0$ for small enough $\eps_0 >0$ to absorb the first term on the right-hand side of  the elliptic inequality and get
$$ \no{\nabla \psi_\eps}{L^4} \lesssim \no{\tau_\eps}{L^2}+1 .$$
Thus $(\nabla \psi_\eps)_\eps$ is uniformly bounded in $L^4(B_{r_1/2}) \cap W^{1,\frac43}(B_{r_1/2})$ and admits a strongly convergent subsequence in $W^{1,2}(B_{r_1/2})$. 

 Multiplying \eqref{absoluteValueEquation} by $1- \rho_\eps$ and integrating by parts over some $B_{r_2}$ with $0<r_2 \leq r_1/2$, we obtain
  \begin{align} \label{densityEstimate}
  \begin{split}
  	\int_{B_{r_2}} |\nabla \rho_\eps|^2 & + \int_{B_{r_2}} \frac{1}{\eps^2} (1- \rho_\eps^2) \rho_\eps (1- \rho_\eps) \\ 
  	& = \int_{\partial B_{r_2} } (1- \rho_\eps) \frac{\partial (1-\rho_\eps)}{\partial r} + \int_{B_{r_2}} \tau_\eps \psi_\eps (1- \rho_\eps) + \int_{B_{r_2}} \rho_\eps (1- \rho_\eps)|\nabla \psi_\eps|^2 \\
  	& \lesssim \int_{\partial B_{r_2}} (1- \rho_\eps) \left| \frac{\partial \rho_\eps}{\partial r} \right| + \left(\no{\tau_\eps}{L^2(B_{r_2})} + \no{\tau_\eps}{L^2(B_{r_2})}^2 +1 \right) \no{1-\rho_\eps}{L^2(B_{r_2})} .
  	\end{split}
  \end{align}
   By  Cavalieri's principle and the mean value theorem
  we see for  some $r_2 \in [r_1/4, r_1/2]$ that
  \begin{align*}
  \int_{\partial B_{r_2}} (1- \rho_\eps) \left| \frac{\partial \rho_\eps}{\partial r} \right| \leq \frac{C}{r_1} \int_{B_{r_2}} (1- \rho_\eps) \left| \frac{\partial \rho_\eps}{\partial r} \right|
  \end{align*}
  holds true. Going back to  inequality \eqref{densityEstimate}  we have
  \begin{align*}
  \int_{B_{r_2}} |\nabla \rho_\eps|^2 \lesssim (\no{\nabla \rho_\eps}{L^2(B_{r_2})} + 1) \no{1-\rho_\eps}{L^2(B_{r_2})} \lesssim \eps ,
  \end{align*}
  which implies $ \rho_\eps \to 1 $ strongly in $W^{1,2}(B_{r_1/4})$.  \\
\textit{Step 4:} Summarizing  the information gathered above, we have in particular
  \begin{align*}
  			\psi_\eps \to \psi \quad & \text{in } W^{1,2}(B_{r_1/4},\R^3), \\
  			\rho_\eps \to \rho\equiv 1  \quad &  \text{in } W^{1,2}(B_{r_1/4})
  \end{align*}
 as well as pointwise a.e. This eventually yields
 \begin{align*}
 					u_\eps = \rho_\eps \psi_\eps  \to \rho \psi =u   
 \end{align*}
in $L^2( B_{r_1/4} , \R^3)$ and since $\rho_\eps = |u_\eps|\leq 1$, we have
\begin{align*}
 					\nabla u_\eps = \psi_\eps \otimes \nabla \rho_\eps + \rho_\eps \nabla \psi_\eps \quad \to
 					\quad \psi \otimes \nabla \rho + \rho \nabla \psi = \nabla u 
 \end{align*}
 in $ L^2(B_{r_1/4},\R^{3 \times 2})$ due to the generalized dominated convergence theorem.
\end{proof}

\subsection{The concentration set $\Sigma_t$}

With respect to Theorem \ref{epsRegularity} we need to determine the properties of the set where strong convergence of $(u_{\eps_k})_k =( d_{\eps_k}(t))_k$ is available. As pointed out previously, strong convergence fails in  finitely many (isolated) points. We define the set of singular points at time $t \in (0,T]$ by 
\begin{align*}
\Sigma_{t} := \bigcap_{0 < r  } \left\{
x_0 \in \To: \liminf_{k \to \infty} \int_{B_r(x_0)} \frac12 |\nabla d_{\eps_k}(t)|^2 + \frac{(1-|d_{\eps_k}(t)|^2)^2}{4 \eps_k^2} > \varepsilon_0^2 \right\}
\end{align*}
where $\varepsilon_0$ is given  by Theorem \ref{epsRegularity}.
\begin{lemma} \label{singularSet}
There exists a constant $K=K(E_0)>0$ independent of $t$ such that   
\begin{align*}
	\# \Sigma_{t} \leq  K
\end{align*}
holds true.
\end{lemma}
\begin{proof}
Choose a finite subset $A_N:=\{ x_l \}_{1\leq l \leq N} \subset \Sigma_{t} $ for $N \in \N$ with mutually disjoint balls $\{ B_{r_l}(x_l) \}_l$. Since $A_N$ is finite, there is a $k_0 \in \N$ such that
\begin{align*}
 \varepsilon_0^2 < \int_{B_{r_l}(x_l)}\frac12 |\nabla d_{\eps_k}|^2 + \frac{(1-|d_{\eps_k}|^2)^2}{4 \eps_k^2} 
\end{align*}
for all $k \geq k_0$ by construction of $\Sigma_{t}$. Thus we have
\begin{align*}
\# A_N = N  < \frac{1}{\varepsilon_0^2} \sum_{l=1}^N\int_{B_{r_l}(x_l)}
 \frac12 |\nabla d_{\eps_k}|^2 + \frac{(1-|d_{\eps_k}|^2)^2}{4 \eps_k^2} \leq \frac{E_0}{\varepsilon_0^2}.
\end{align*}
due to the energy estimate \eqref{energyLaw}.
By the  arbitrariness of $A_N $,  the set $\Sigma_{t} $  consists of at most $K:=\left\lceil \frac{E_0}{\varepsilon_0^2}\right\rceil$ points.
\end{proof}
As a consequence of the previous lemma,  we  find a subsequence of $(d_{\eps_k}(t))_k$ strongly converging on $\To\backslash \Sigma_t$ for every $t$ under consideration.
\begin{lemma} \label{localConvergence}
For $t \in (0,T]$ let $(u_{\eps_k})_k = ( d_{\eps_k}(t))_k$. Then there exists a subsequence\footnote{not relabeled} such that
$$ \nabla d_{\eps_k}(t) \to \nabla d(t)$$
in $L^2_{\loc}(\To \backslash \Sigma_{t},\R^{3 \times 2})$.
\end{lemma}
\begin{proof}
Let  $(z_j)_{j \in \N}$ be the set of rational points on $  \To \backslash \Sigma_{t} $
 and define
\begin{equation} r_j:= \sup \left\{ r>0: \liminf_{k \to \infty} \int_{B_r(z_j)} \frac12 |\nabla d_{\eps_k}(t)|^2 + 
	\frac{(1-|d_{\eps_k}(t)|^2)^2}{4 \eps_k^2} \leq \varepsilon_0^2 \right\}. \label{supRadii}
	\end{equation}
In general the radii $r_j$ might be  too large to satisfy
$$\int_{B_{r_j}(z_j)} \frac12 |\nabla d_{\eps_k}(t)|^2 + 
	\frac{(1-|d_{\eps_k}(t)|^2)^2}{4 \eps_k^2} \leq \varepsilon_0^2,$$
 which is why we consider  $\frac45 r_j$. In view of Theorem \ref{epsRegularity} we want to prove  $ \bigcup_{j \in \N}  B_{r_j/5}(z_j) = \To \backslash \Sigma_{t}$.\\
Let $z \in \To \backslash \Sigma_t$ with $r_z>0$  be such
that 
$$ \liminf_{k \to \infty} \int_{B_{r_z}(z)} \frac12 |\nabla d_{\eps_k}(t)|^2 + 
	\frac{(1-|d_{\eps_k}(t)|^2)^2}{4 \eps_k^2} \leq \varepsilon_0^2.$$
By density, we choose a $z_{j_0}$ such that $|z_{j_0}-z|< \frac{r_z}{6}$. Thus we have $r_{j_0} \geq \frac56 r_z$ from
\eqref{supRadii} and therefore $|z_{j_0}-z| < \frac{r_{j_0}}{5}$.

Since the covering is countable we use a diagonal argument, Theorem \ref{epsRegularity} and Lemma 3.2 to extract a subsequence which fulfills the assertion.
\end{proof}

\subsection{Limit passage  $\eps \to 0^+$}
 Finally, exploiting the results from the previous section, we want to conclude the desired limit passage as $
\eps \to 0^+$. As mentioned in the introduction, the idea is to look at fixed times $t$ and use concentration-cancellation arguments from \cite{diperna} for the stress tensor term $\div (\nabla d \odot \nabla d)$ in the momentum equation. 

The inspiration for this argument is taken from  \cite{harpes}, which  in turn relies on \cite{struwe}. Well-known arguments from  \cite{chen2} (see also \cite{alouges}) then allow to pass from \eqref{approx2} to \eqref{equationLLG}.

\begin{proof}[Proof of Theorem 2.1] According to Section 3.1, we take a subsequence\footnote{not relabeled} of solutions to \eqref{approx1}--\eqref{approx2}
$(v_\eps, d_\eps)_\eps$ with $\eps \to 0^+$ such that \eqref{conv1}--\eqref{conv9} holds true. 

To begin with, we consider the weak formulation of \eqref{approx2}. Note that for $\xi \in W^{1,2}(\To \times [0,T], \R^3) \cap L^\infty( \To \times [0,T] , \R^3)$ we have that $ d_\eps \wedge \xi$ is a proper test function. The identity 
$d_\eps \wedge \Delta d_\eps = \div ( d_\eps \wedge \nabla d_\eps)$ then yields
\begin{align*}
\int_0^T  \int_\To \big( d_\eps \wedge ( \partial_t d_\eps + (v_\eps \cdot \nabla) d_\eps)    \big) \cdot \xi + 
\int_0^T \int_\To  ( d_\eps \wedge  \nabla d_\eps ) : \nabla \xi =0,
\end{align*}
where  we wrote
$$( a \wedge \nabla b ) : \nabla c = \sum_j  (a \wedge \partial_{x_j} b ) \cdot \partial_{x_j} c$$
for $a,b,c : \To \to \R^3$ being weakly differentiable. From the convergence statements \eqref{conv1}, \eqref{conv4}-- \eqref{conv7} and the bound from the maximum principle $|d_\eps |\leq 1$ a.e.\ for all $\eps >0$, we conclude that the limit of the weak formulation is
\begin{align*}
\int_0^T  \int_\To \big( d \wedge ( \partial_t d + (v \cdot \nabla) d)    \big) \cdot \xi + 
\int_0^T \int_\To  ( d \wedge  \nabla d ) : \nabla \xi =0.
\end{align*}
Here we have $|d|\equiv 1$ a.e.\ therefore all derivatives (in particular the first term involving $\partial_t d$ and $\partial_{x_j}  d$ for $j=1,2$) are perpendicular to $d$ a.e. Using this fact and setting  $\xi = d \wedge \Phi$ with $\Phi \in C^\infty( \To \times [0,T],\R^3)$, we obtain
\begin{align*}
\int_0^T  \int_\To \left( \partial_t d + (v \cdot \nabla) d \right) \cdot \Phi + 
\int_0^T \int_\To   \nabla d  : \nabla \Phi - \int_0^T \int_\To |\nabla d|^2 d \cdot \Phi =0,
\end{align*}
employing the Lagrange identity $ (a \wedge b)\cdot  (c \wedge d) = (a\cdot c) (b\cdot d) - (b\cdot c) (a \cdot d)$ for the wedge-product. This shows that the limits $(v,d)$ satisfy the director equation \eqref{equationLLG}. 

The remaining part is to show that $(v,d)$ fulfill the momentum equation. As $
\nabla d_\eps(t) \weak \nabla d(t)$, Theorem \ref{epsRegularity} and Proposition \ref{singularSet} show that no oscillations occur in the limit and the set of concentrations is finite at least for a subsequence.
 \\
Recall from \eqref{conv8} and \eqref{conv9} that $(v_\eps,d_\eps)(\cdot , t)$ strongly converges to the limit for a.e. $t \in [0,T]$. We set $\tau_\eps := \partial_t d_\eps + (v_\eps \cdot \nabla) d_\eps$ and $\tau := \partial_t d + (v \cdot \nabla)d$. Due to the a-priori bounds from Section 3.1, the set  
$$ A :=\left\{ t \in [0,T]: \liminf_{\eps \to 0^+} \left( \no{\partial_t v_\eps(t)}{X_r^*} + \no{ \nabla v_\eps(t)}{L^2} + \no{\nabla d_\eps(t)}{L^2} + \no{ \tau_\eps(t) }{L^2} \right)< \infty \right\}$$
has full measure by Fatou's lemma, i.e.\ $|A|=T$. Without loss of generality, let $A$ be such that $(v_\eps,d_\eps)(t) \to (v,d)(t)$ as in \eqref{conv8} and \eqref{conv9} for every $t \in A$. Fix $t \in A$. Now there exists a subsequence for which
\begin{equation*}
\left( \partial_t v_{\eps_j}, \nabla v_{\eps_j}, \nabla d_{\eps_j} , \tau_{\eps_j} \right)(t)
~ \weak ~ \left(\partial_t v, \nabla v , \nabla d , \tau \right)(t),
\end{equation*} 
where we identified the limit in $t$ by the strong convergence of $(v_{\eps_j}(t),d_{\eps_j}(t))_{j \in \N}$ in $L^2$. Since this is true for
any subsequence, the full sequence $\left( ( \partial_t v_\eps , \nabla v_\eps  , \nabla d_\eps, \tau_\eps )(t)\right)_\eps$ converges weakly.

Next, we take a  test function $\phi \in C_{\div}^\infty(\To,\R^2)$. Since $\phi$ is solenoidal, it is
 $$ \phi = \nabla^\perp \eta + \text{const.}= (  - \partial_2 , \partial_1 )^\top \eta + \text{const.}$$
 for some $\eta \in C^\infty(\To)$ (see Section 2) and the constant vanishes if $ \int_\To \phi =0$. Testing the momentum equation \eqref{approx1} at time $t\in A \backslash \{0\}$ by $\phi$ we obtain
\begin{equation}	\label{weakInSpace}
\begin{split}
\int_\To \partial_t v_\eps (t) \cdot \phi + \int_\To v_\eps(t) \otimes v_\eps(t) : \nabla \phi  + 
\int_\To \nabla v_\eps(t) : \nabla \phi \\
- \int_\To \nabla d_\eps(t) \odot \nabla d_\eps(t) : 
\begin{pmatrix} - \partial_1 \partial_2 \eta & -\partial_2^2 \eta \\ \partial_1^2 \eta & \partial_1 \partial_2 \eta \end{pmatrix}
=0.
\end{split}
\end{equation}
By Lemma \ref{localConvergence}, there exists a subsequence\footnote{not relabeled} $(v_\eps, d_{\eps})_\eps$, which in general depends on $t$, such that 
$$ \nabla d_{\eps}(t) \to \nabla d(t)$$
in $L^2_{\loc}(\To \backslash \Sigma_{t},\R^{3 \times 2})$, where $\Sigma_{t}$ is finite according to Lemma 3.2. 

By density, it suffices to show the weak formulation \eqref{weakInSpace} in the limit for all functions $\eta(x) = e^{ik\cdot x}$ with 
$k \in \dot{\Z}^2$. First note that the only problematic terms are the ones related to $\partial_t v$ and $\nabla d \odot \nabla d$.  However, choosing a smooth cut--off function $\psi$ which vanishes in a neighborhood of $\Sigma_{t}$, we may pass to the limit with the test function $\nabla^\perp \eta(x)= \nabla^\perp \left( e^{ik\cdot x} \psi(x) \right)$, i.e.
\begin{equation}	\label{weakWithHoles}
\begin{split}
 \la \partial_t  v (t) , \nabla^\perp \eta \ra_{X_r^*,X_r}  & + \int_\To v(t) \otimes v(t) : \nabla \nabla^\perp \eta  + 
\int_\To \nabla v(t) : \nabla \nabla^\perp \eta  \\
&- \int_\To \nabla d(t) \odot \nabla d(t) : \nabla \nabla^\perp \eta
=0.
\end{split}
\end{equation}
It remains  to "fill" the holes and we do so by considering every point in $\Sigma_{t}$ separately.
To this end,  observe  
 that the equations \eqref{equationNS} and \eqref{approx1} are covariant under rotations. To be more specific, let $0 = x_0 \in \Sigma_t$, without loss of generality, be the  invariant point of the rotation. Taking a test function $\nabla^\perp \beta$ with $\operatorname{supp}\beta \subset B_r$, $r>0$, the rotation of coordinates 
$Qy=x$ for $Q^\top =Q^{-1}$ yields 
\begin{equation} \label{trafo}
 \int_{B_r} (\nabla d_\eps \odot \nabla d_\eps)(x,t) : \nabla \nabla^\perp \beta(x,t)
 \, \d x = \int_{B_r} (\nabla_y d_\eps \odot \nabla_y d_\eps)(Qy,t) : \nabla_y \nabla_y^\perp
 \beta(Qy,t) \, \d y
\end{equation}
by a change of variables. Similar identities hold for all other terms, i.e. 
$$ \int f_n (x) \phi (x) \, \d x \to \int f(x) \phi (x) \, \d x \qquad \text{iff} \qquad \int f_n (Qy) \phi (Qy) \, \d y \to \int f(Qy) \phi (Qy) \, \d y.$$
For $r>0$ sufficiently small we know by  Lemma \ref{localConvergence} that $(\nabla d_\eps \odot \nabla d_\eps)(\cdot ,t)$ only concentrates in $x_0=0 \in B_r$ and so does $(\nabla_y d_\eps \odot \nabla_y d_\eps)(Q(\cdot),t)$ by the same token. We imitate the key idea of \cite{diperna}. Because of the rotational covariance, it is enough to consider test functions $h(x)=h(x_1)$ in a neighborhood of the concentration point $x_0=0$. 
To cut off the concentration point, define $h_n$ for $n\in \N$ large enough by the elliptic  ODE
$$ h_n'' =(1-  \mathds{1}_{(-1/n,1/n)} ) h'', \qquad h_n(-r) = h(-r),~ h_n(r) = h(r).$$
We properly localize the function $h_n$ by considering $\eta_n(x_1,x_2)=h_n(x_1) \chi(x_1,x_2)$ with $\chi$ being smooth, $\chi\equiv 1$ on $B_{r/2}$ and zero outside of $B_r$. We set $\eta = h \chi$ respectively and note that
\begin{align*}
\nabla^2 \eta_n = 	\nabla^2 (h_n \cdot \chi) \to  \nabla^2 (h \cdot  \chi)= \nabla^2 \eta    \qquad \text{almost everywhere on }\To
\end{align*} 
and by dominated convergence in any $L^p(\To)$, $1\leq p<\infty$; therefore $\eta_n \to \eta $ in $W^{2,p}(\To)$  (in particular $\nabla^\perp  \eta_n\to \nabla^\perp \eta$ in $X_s=W_\div^ {1,s}(\To,\R^2)$ for $2<s<\infty$).

Choosing $\phi= \nabla^\perp \eta_n$ in \eqref{weakInSpace}, we are able to pass to the limit in $\eps$ since $\nabla \nabla^\perp \eta_n$ vanishes around the concentration point. The limit then reads
\begin{equation} \label{nEquation}
\begin{split}
 & \la \partial_t  v (t) , \nabla^\perp \eta_n \ra_{X_r^*, X_r}   + \int_{B_r} v(t) \otimes v(t) : \nabla \nabla^\perp \eta_n  + 
\int_{B_r} \nabla v(t) : \nabla \nabla^\perp \eta_n  \\
&-\int_{B_r \backslash B_{r/2}} \nabla d(t) \odot \nabla d(t) : \nabla \nabla^\perp \eta_n
- \int_{B_{r/2}} [\nabla d \odot \nabla d]_{2,1} h_n''
=0.
\end{split}
\end{equation}
with $[A]_{ij}=a_{ij}$ for  $A=(a_{ij})_{ij}\in \R^{M \times N}$.
As $n \to \infty$ we are able to replace $\eta_n$ by $\eta$ in the second, third and fourth term due to $v\otimes v \in L^2(\To,\R^2), \nabla v \in L^2(\To, \R^{2 \times 2}) , \nabla d \odot \nabla d \in L^1(\To\backslash B_{r/2},\R^{2 \times 2})$ from \eqref{conv1}--\eqref{conv6}. For the first term we have
$$  \la \partial_t  v (t) , \nabla^\perp \eta_n \ra_{X_s^*,X_s} \quad \to \quad  \la \partial_t  v (t) , \nabla^\perp \eta \ra_{X_s^*,X_s}$$
since $\partial_t v(t) \in X_s^*$ and $ \nabla^\perp \eta_n \to \nabla^\perp \eta $ in $X_s$. In order to use Lebesgue's dominated convergence theorem for the last term of \eqref{nEquation}, we observe that
\begin{align*}
[\nabla d \odot \nabla d]_{2,1} h_n'' \quad &\to \quad  [\nabla d \odot \nabla d]_{2,1} h'' \quad  \text{a.e. } \\
 \big|[\nabla d \odot \nabla d]_{2,1} h_n''\big| \quad &\leq \quad \big|[\nabla d \odot \nabla d]_{2,1} h''\big| \quad \in L^1(B_{r/2})
\end{align*}
is valid.  This and $\nabla \nabla^\perp \eta=h''$ on $B_{r/2}$ yield \eqref{weakWithHoles} for $\eta = h \chi$. \\
 By \eqref{trafo}, we deduce that the weak formulation is also satisfied for test functions of the form $\nabla^\perp \eta(x) =\nabla^\perp \left( e^{ik\cdot x} \chi(x)\right)$, $k \in \dot{\Z}^2$, where $\chi$ is a proper chosen cut-off function around a concentration point. Combining this with \eqref{weakWithHoles} and using the density ($W^{3,2}(\To)$ is enough) of $\{e^{ik\cdot(\cdot )}: k \in \Z^2\}$ in the space of test functions, we  eventually obtain that $(v,d)$ satisfy the weak formulation 
\begin{equation}	\label{weakSpace}
\begin{split}
 \la \partial_t  v (t) , \phi \ra_{X_r^*,X_r}  & + \int_\To v(t) \otimes v(t) : \nabla \phi  + 
\int_\To \nabla v(t) : \nabla  \phi  \\
&- \int_\To \nabla d(t) \odot \nabla d(t) : \nabla \phi
=0
\end{split}
\end{equation}
for every $\phi \in C_{\div}^\infty(\To,\R^2)$ and  $t \in A$.  As $t$ was arbitrary and $A$ has full measure, \eqref{weakSpace} holds for a.a. $t \in (0,T]$.\\
In order to deal with the time dependence we multiply \eqref{weakSpace} by $\zeta(t)$ with $\zeta \in C^\infty([0,T])$ with $\zeta(T)=0$ and integrate over $[0,T]$. 
The density of $C^\infty_\div(\To,\R^2) \otimes C^\infty([0,T])$ in $C^\infty_\div(\To \times [0,T],\R^2)$  yields
\begin{align*}
\int_0^T \la \partial_t v , \phi \ra_{X_r^*,X_r} &+ \int_0^T \int_\To v \otimes v : \nabla \phi  + 
\int_0^T \int_\To \nabla v : \nabla \phi \\
&-\int_0^T \int_\To \nabla d \odot \nabla d : \nabla \phi 
=0
\end{align*}
for all $\phi \in C_{\div}^\infty(\To \times [0,T],\R^2)$ and $\phi(T)=0$. Although the regularity of 
$\partial_t v$ is too weak to use an integration by parts formula, we know from 
$$ \int_0^T \int_\To \partial_t v_\eps \cdot  \phi =  -\int_\To v_0 \cdot  \phi(0) - 
\int_0^T \int_\To v_\eps \cdot  \partial_t \phi$$
for every $\eps >0$  that we have
 $$ \int_0^T  \la \partial_t v, \phi \ra_{X_r^*,X_r} =  -\int_\To v_0 \cdot \phi(0) - 
\int_0^T \int_\To v  \cdot \partial_t \phi.$$
according to \eqref{conv1} and \eqref{conv3}. \\
From \eqref{conv1}--\eqref{conv7} we gain an improvement of regularity, i.e.\ $(v,\nabla d) \in C_w ([0,T]; L^2(\To, \R^2 \times \R^{3 \times 2} ) )$.
In particular, the solution $(v,d)$ attains the  initial data $(v_0,d_0)$.
Hence  \eqref{equationNS}--\eqref{ini4} possesses a weak solution in the sense of Definition 2.1. 
 \end{proof}
 \begin{proof}[Proof of Theorem 2.2]
 The existence of a weak solution follows from Theorem 2.1. The energy inequality 
 follows from \eqref{energyLaw} and \eqref{conv1}--\eqref{conv9} as well as the lower semicontinuity of the norms with respect to  weak convergence.
\end{proof}

\bigskip

\noindent{\bf Acknowledgments:}  
The author would like to thank his supervisor Anja Schl\"omerkemper  for giving the opportunity of working on this  topic and for her support in writing this article. Also, the author would  like to thank Francesco De Anna for fruitful discussions on this topic, Martin Kalousek for reading an earlier version of the manuscript and Jesse Ratzkin for pointing out a mistake in an earlier version.


\begin{thebibliography}{19}
\baselineskip=12pt
{\footnotesize


\bibitem{alouges}
Alouges, F., Soyeur, A.: On global weak solutions for Landau-Lifshitz equations: existence and nonuniqueness. {\it Nonlinear Anal.} {\bf 18} (1992), 1071--1084.

\bibitem{ball2017}
Ball, J.M.: Mathematics and liquid crystals. {\it Mol. Cry. Liq. Cry.} {\bf 647}(1) (2017), 1--27.

%

\bibitem{bertsch}
Bertsch, M., Dal Passo, R., van der Hout, R.: Nonuniqueness for the Heat Flow of Harmonic Maps on the Disk. {\it Arch. Ration. Mech. Anal.} {\bf 161} (2002), 93--112.

\bibitem{bethuel1994}
Bethuel, F., Brezis, H., Helein, F.: {\it Ginzburg-Landau vortices}, 
Birkhäuser, Boston, 1994.

\bibitem{chen}
 Chen, Y.: Weak solutions to the evolution problem for harmonic maps into spheres. {\it  Math. Z.} {\bf 201} (1989), 69--74.

\bibitem{chen2}
 Chen, Y., Struwe, M.: Existence and partial regularity results for the heat flow of harmonic maps. {\it Math. Z.} {\bf 201} (1989), 83--103.
 
  \bibitem{deanna}
De Anna, F., Liu, C.: Non-isothermal General Ericksen–Leslie System: Derivation, Analysis and Thermodynamic Consistency. {\it Arch. Ration. Mech. Anal.}  {\bf 231} (2019), 637--717.
 
 \bibitem{debouard}
De Bouard, A., Hocquet, A., Prohl, A.: {\it Existence, uniqueness and regularity for the stochastic Ericksen-Leslie equation.} arXiv:1902.05921v1.
 
 \bibitem{delort}
Delort, J.M.: Existence de nappes de tourbillon en dimension deux. {\it J. Amer. Math. Soc.} {\bf 4} (1991),
553--586.
 
 \bibitem{diperna}
DiPerna, R., Majda, A.: Reduced Hausdorff dimension and concentration-cancellation for two-dimensional incompressible flow. {\it J. Amer. Math. Soc.} {\bf 1} (1988), 59--95.
 
\bibitem{ericksen}
Ericksen, J.L.: Hydrostatic theory of liquid crystals. {\it Arch. Ration. Mech. Anal.} {\bf 9} (1962), 371--378
 
\bibitem{evans}
Evans, L.C., M\"uller, S.: Hardy spaces and two-dimensional Euler equations with nonnegative vorticity. {\it J. Amer. Math. Soc.} {\bf 1} (1994), 199--219.
 
\bibitem{feng}
Feng, Z., Hong, M.C., Mei,Y.: {\it Convergence of the Ginzburg-Landau approximation for the Ericksen-Leslie system.}
arXiv:1804.08203v2.

 
\bibitem{gilbarg}
Gilbarg, D., Trudinger, N.S.: {\it Elliptic Partial Differential Equations of Second Order}, 2nd.~ed., Springer Berlin Heidelberg, 2001.

\bibitem{harpes}
Harpes, P.: Partial compactness for the 2-D Landau-Lifshitz flow. {\it Electron. J. Differential Equations} {\bf 90} (2004), 1--24.

\bibitem{hieber}
Hieber, M.G., Pr\"uss, J.W.: Modeling and analysis of the Ericksen-Leslie equations for nematic liquid crystal flows.
{\it Handbook of Mathematical Analysis in Mechanics of Viscous Fluids}, Springer International Publishing 2018, 
1075--1134.

\bibitem{hong}
Hong, M.C.: Global existence of solutions of the simplified Ericksen--Leslie system in dimension two.
{\it Calc. Var. Partial Differential Equations} {\bf 40} (2010), 15--36.

\bibitem{lin2016}
Huang, T., Lin, F.H., Liu, C., Wang, C.: Finite time singularity of the nematic liquid crystal flow in dimension three. 
{\it Arch. Ration. Mech. Anal.} {\bf 221} (2016), 1223--1254.

\bibitem{leslie}
Leslie, F.M.: Some constitutive equations for liquid crystals. {\it Arch. Ration. Mech. Anal.} {\bf 28} (1968), 265--283.

\bibitem{lin}
Lin, F.H., Lin, J., Wang, C.Y.: Liquid crystal flows in two dimensions. {\it Arch. Ration. Mech. Anal.} {\bf 197} (2010), 297--336.

\bibitem{lin1989}
Lin, F.H.: Nonlinear theory of defects in nematic liquid crystal: phase transition and
flow phenomena. {\it Comm. Pure Appl. Math.} {\bf 42} (1989), 789--814.

\bibitem{lin1995}
Lin, F.H., Liu, C.: Nonparabolic dissipative systems modeling the flow of liquid crystals. {\it Comm. Pure Appl. Math.} {\bf 48} (1995), 501--537.

\bibitem{lin1999}
Lin, F.H., Wang, C.Y.: Harmonic and quasi-harmonic spheres. {\it Comm. Anal. Geom.} {\bf 7}(2) (1999), 397--429.
 
\bibitem{lin2}
Lin, F.H., Wang, C.Y.: Global existence of weak solutions of the nematic liquid crystal flow in dimension three {\it Comm. Pure Appl. Math.} {\bf 69} (2016), 1532--1571.

\bibitem{majda2002}
Majda, A.J., Bertozzi, A.L.: {\it Vorticity and incompressible flow},
Cambridge Texts in Applied Mathematics, Cambridge University Press, 2002.

\bibitem{robinson2016}
Robinson, J.C., Rodrigo, J.L, Sadowski, W.: {\it The Three-Dimensional Navier--Stokes Equations}, Cambridge University Press, 2016.

\bibitem{qing}
Qing, J.: On singularities of the heat flow for harmonic maps from surfaces into spheres. {\it Comm. Anal. Geom.} {\bf 3} (1995), 297--315.

\bibitem{roubicek}
Roub{\'\i}{\v{c}}ek, T.: {\it Nonlinear Partial Differential Equations with Applications},
International Series of Numerical Mathematics, Springer Basel, 2013.

\bibitem{schloemerkemper2018}
Schl\"omerkemper, A., {\v{Z}}abensk{\'{y}}, J.: Uniqueness of solutions for a mathematical model for magneto-viscoelastic flows. {\it Nonlinearity} {\bf 31}(6), (2018), 2989--3012.

\bibitem{schochet}
Schochet, S.: The weak vorticity formulation of the 2D Euler equations and
concentration--cancellation. {\it Comm. Partial Differential Equations} {\bf 20} (1995), 1077--1104.

\bibitem{struwe}
Struwe, M.: On the evolution of harmonic mappings of Riemannian surfaces. {\it Comment. Math. Helv.}
{\bf 60} (1985), 558--581.

\bibitem{struwe2}
Struwe, M.: On the evolution of harmonic maps in higher dimensions. {\it J. Differential Geom.} {\bf 3} (1988),
485--502.

\bibitem{topping}
Topping, P.: Reverse bubbling and nonuniqueness in the harmonic map flow. {\it Int. Math. Research Notes} {\bf 10}
(2002), 505--520.

\bibitem{walkington}
Walkington, N.J.: Numerical approximation of nematic liquid crystal flows governed by the Ericksen-Leslie equations.
{\it ESAIM: Math. Model. Numer. Anal.} {\bf 45}(3) (2011), 523--540.

\bibitem{wang} 
Wang, C.Y.: Bubble phenomena of certain Palais-Smale sequences from surfaces to general targets. {\it Houston J. Math.} {\bf 22}(3) (1996), 559---590.
}

\end{thebibliography}
\end{document}